\documentclass[12pt,leqno]{amsart}
\usepackage{amsmath,amsfonts,latexsym,graphicx,amssymb,url, color}
\usepackage{hyperref}
\usepackage{txfonts} 
\newcommand{\average}{{\mathchoice {\kern1ex\vcenter{\hrule
height.4pt width 6pt depth0pt} \kern-9.7pt}
{\kern1ex\vcenter{\hrule height.4pt width 4.3pt depth0pt}
\kern-7pt} {} {} }}

\newcommand{\abs}[1]{\left\vert#1\right\vert}

\newcommand{\trace}{\text{trace}}
\newcommand{\R}{{\mathbb R}} 

\newcommand{\e}{\varepsilon} 

\newcommand{\calF}{{\mathcal F}}
\newcommand{\calO}{{\mathcal O}}

\newcommand{\calM}{{\mathcal M}}

\newcommand{\bh}{{\bar h}}

\newcommand{\p}{\partial}

\newenvironment{myindentpar}[1]%
{\begin{list}{}%
         {\setlength{\leftmargin}{#1}}%
         \item[]%
}
{\end{list}}
\theoremstyle{plain}
\newtheorem{theorem}{Theorem}[section]
\newtheorem{corollary}[theorem]{Corollary}
\newtheorem{lemma}[theorem]{Lemma}
\newtheorem{proposition}[theorem]{Proposition}

\begin{document}
\date{May 31, 2012}
\title[Boundary sections of the Monge--Amp\`ere equation]{Geometric properties of 
boundary sections of solutions to the Monge--Amp\`ere equation and applications}
\author{Nam Q. Le}
\address{Department of Mathematics, Columbia University, New York, NY 10027, USA\\ and School of Engineering, Tan Tao University, Long An, Vietnam}
\email{\tt  namle@math.columbia.edu, nam.le@ttu.edu.vn}
\author{Truyen Nguyen}
\address{Department of Mathematics, The University of Akron, Akron, OH 44325, USA}
\email{tnguyen@uakron.edu}

\begin{abstract}
In this paper, we establish several geometric properties of boundary 
sections of convex solutions to the Monge-Amp\`ere equations: the engulfing and separating properties and volume 
estimates. As applications, we prove a  covering lemma of Besicovitch type, a covering theorem and a 
strong type $p-p$ estimate for the maximal function corresponding to boundary sections. Moreover, we show that the
Monge-Amp\`ere setting forms a space of homogeneous type.  
\end{abstract}
\maketitle

\setcounter{equation}{0}

\section{Introduction}
 In recent years, there has been a growing interest in studying boundary regularity of solutions to the Monge-Amp\`ere equation and its linearization, that is, 
the linearized Monge-Amp\`ere equation. Solutions of many important problems in Analysis and Geometry require a deep understanding of boundary behaviors of the above 
Monge-Amp\`ere type equations. Among those, one can mention the problems of global regularity of the affine maximal surface equation \cite{TW1, TW2, L} and  
Abreu's equation in the context of existence of K\"ahler metric of constant scalar curvature \cite{D1, D2, D3, D4, Zh}. In these papers, the properties of boundary 
cross sections of solutions to the Monge-Amp\`ere equation play an important role as those of interior sections in the landmark paper \cite{C2} where Caffarelli 
discovered surprising interior $W^{2, p}$ estimates for solutions to the Monge-Amp\`ere equation with  right hand side being continuous and bounded away from $0$ and $\infty$.\\

The notion of sections (or cross sections) of solutions to the Monge-Amp\`ere equation
\begin{equation}\label{MA}
\det D^2\phi = g \quad \mbox{in}\quad \Omega
\end{equation}
 was first introduced and studied by Caffarelli \cite{C1,C2,C3,C4}. Sections are defined 
as sublevel sets of convex solutions after subtracting their supporting hyperplanes. Understanding the geometry of sections is essential in obtaining sharp regularity properties for 
solutions of \eqref{MA}. As a matter of fact, the structure of equation \eqref{MA} is ultimately related to that of sections of its solutions, and by studying the shape of interior sections Caffarelli derived fundamental interior regularity estimates for \eqref{MA} in the above mentioned papers. When the right hand side of the Monge-Amp\`ere equation is only bounded, sections of 
 solutions in the sense of Aleksandrov can have degenerate geometry. However, in many applications in Analysis and Geometry involving equations of Monge-Amp\`ere type, we would like these sections to have properties similar to Euclidean balls as in uniformly elliptic equations. This is the case of interior sections through the work of  Caffarelli \cite{C1,C3, C4}, Caffarelli-Guti\'errez \cite{CG1, CG2} and Guti\'errez-Huang \cite{GH}. The case of boundary sections is less well understood. However, thanks to Savin's Localization theorem \cite{S1, S2} at the boundary for solutions of \eqref{MA}, we expect many properties of interior sections hold also for boundary ones. This is the subject of our present paper.\\

The purpose of this paper is to investigate several important geometric properties of  boundary sections of convex solutions to the Monge-Amp\`ere equation \eqref{MA} 
with right hand side bounded away from $0$ and $\infty$ and with smooth boundary data: engulfing and separating 
properties, and volume estimates. As applications, we prove a 
 covering lemma of Besicovitch type and employ it to prove a covering theorem  and a strong type $p-p$ estimate for 
the maximal function with respect  to boundary sections.  Moreover,  we introduce a quasi-distance induced by boundary 
sections and show that the structure of our Monge-Amp\`ere equation 
 gives rise to a space of homogeneous type.
This allows us to place the  Monge-Amp\`ere setting  in a more general context where many real analytic problems have 
been studied, see  \cite{CW, DGL}. 
Our results are boundary version of those established by Caffarelli-Guti\'errez \cite{CG1, CG2}, Guti\'errez-Huang \cite{GH} and Aimar-Forzani-Toledano \cite{AFT} 
for interior sections of solutions to  equation \eqref{MA} (see also the book by Guti\'errez \cite{G}). The results in this paper are crucial for our studies in \cite{LN1, LN2} about boundary regularity for solutions to the
linearized  Monge-Amp\`ere equation 
\begin{equation}\label{LMA}
\trace(\Phi D^2 u)
= f \quad \mbox{in}\quad \Omega,
\end{equation}
where $\Phi := (\det D^2 \phi) ~ (D^2\phi)^{-1}$  with $\phi$ being a convex solution of \eqref{MA}. In \cite{LN1, LN2}, we investigate  equation \eqref{LMA}  and establish 
global $W^{2,p}$  and $W^{1, p}$ estimates for its solutions which are  boundary version of interior estimates obtained recently in \cite{GN1,GN2}.\\

The rest of the paper is organized as follows. We state our main results in Section~\ref{main_res}. In Section~\ref{data_sec}, 
we recall the main tool to study geometric properties of boundary sections of solutions to the Monge-Amp\`ere equation: 
the Localization theorem at the boundary for solutions.
Geometric properties of boundary sections are established in Section~\ref{geo_sec}. 
In Section~\ref{covering_sec}, we prove a Besicovitch-type covering lemma, a covering theorem and a strong type $p-p$ estimate for the maximal function corresponding to boundary sections.
Finally, in Section~\ref{sec:quasi-distance}, we 
 show that the Monge-Amp\`ere  setting forms a space of homogeneous type.
 
\section{Statement of the main results}
\label{main_res}
The results in this paper hold under the following global information on the convex domain $\Omega$ and the convex function $\phi$.
We assume that 
\begin{equation}
\Omega\subset B_{1/\rho} ~\text{contains an interior ball of radius $\rho$ tangent to}~ \p 
\Omega~ \text{at each point on} ~\p \Omega.
\label{global-tang-int}
\end{equation}
Let $\phi : \overline \Omega \rightarrow \R$, $\phi \in C^{0,1}(\overline 
\Omega) 
\cap 
C^2(\Omega)$  be a convex function satisfying
\begin{equation}\label{eq_u}
\det D^2 \phi =g, \quad \quad 0 <\lambda \leq g \leq \Lambda \quad \text{in $\Omega$}.
\end{equation}
Assume further that on $\p \Omega$, $\phi$ 
separates quadratically from its 
tangent planes,  namely
\begin{equation}
\label{global-sep}
 \rho\abs{x-x_{0}}^2 \leq \phi(x)- \phi(x_{0})-\nabla \phi(x_{0}) \cdot (x- x_{0})
 \leq \rho^{-1}\abs{x-x_{0}}^2, ~\forall x, x_{0}\in\p\Omega.
\end{equation}
The section of $\phi$ centered at $x\in \overline \Omega$ with height $h$ is defined by
\begin{equation*}
 S_{\phi} (x, h) :=\big\{y\in \overline \Omega: \quad \phi(y) < \phi(x) + \nabla \phi(x)\cdot (y- x) +h\big\}.
\end{equation*}
For  $x\in \Omega$, we denote by $\bar{h}(x)$ the maximal height of all sections of $\phi$ centered at $x$ and contained in $\Omega$, that is,
$$\bar{h}(x): =\sup\{h\geq 0| \quad S_{\phi}(x, h)\subset \Omega\}.$$
In this case, $S_{\phi}(x, \bar{h}(x))$  is called the maximal interior section of $\phi$ with center $x\in\Omega$. 
In what follows, we will drop the dependence on $\phi$ of sections when no confusion arises.\\

We denote by $c, \bar{c}, C, C_{1}, C_{2}, \theta_{0}, \theta_{\ast}, \cdots$, positive constants depending only on $\rho$, $\lambda$, $\Lambda$, 
$n$, and their values may change from line to line whenever 
there is no possibility of confusion. We refer to such constants as {\it universal constants}.\\

Our first result is the engulfing property of sections $\{S(x, t)\}$. 
\begin{theorem}
\label{engulfing2}
Assume that the convex domain $\Omega$ and the convex function $\phi$ satisfy \eqref{global-tang-int}--\eqref{global-sep}. There exists $\theta_{\ast}>0$ depending only 
on $\rho, \lambda, \Lambda$ and $n$ such that if $y\in S(x, t)$ with $x\in\overline{\Omega}$ and $t>0$, then $S(x, t)\subset S(y, \theta_{\ast}t).$
\end{theorem}
In this case, we say that the sections $\{S(x, t)\}$ of $\phi$  satisfy the engulfing property with the constant $\theta_{\ast}$.\\

The engulfing property of sections will be shown to be equivalent to the separating property of sections as stated in the following.
\begin{proposition}
\label{separating}
(i) Assume that the convex domain $\Omega$ and the convex function $\phi$ satisfy \eqref{global-tang-int}--\eqref{global-sep}.
 Let $\theta_{\ast}$ be the constant in 
 Theorem~\ref{engulfing2}. Then,  the sections $\{S(x, t)\}$ of $\phi$ satisfy the separating property with the constant $\theta_{\ast}^2$, namely, if $y\not\in S(x, t)$, then
$$S(y, \frac{t}{\theta_{\ast}^2})\cap S(x, \frac{t}{\theta_{\ast}^2})=\emptyset.$$
(ii) Conversely, assume that the sections $\{S(x, t)\}$ of a convex function $\phi$ defined on a convex domain $\Omega$ satisfy the separating property with the constant $\theta$. Then the sections $\{S(x, t)\}$ satisfy the engulfing property with the constant $\theta^2$.
\end{proposition}
A key in the proof of Theorem \ref{engulfing2} is a dichotomy for sections of solutions to the Monge-Amp\`ere equation: any section is either an interior section or included in a boundary section with comparable height. Thus, when dealing with sections, we can focus our attention to only interior sections and boundary sections. The precise statement is as follows.

\begin{proposition} 
\label{dicho}
Assume that the convex domain $\Omega$ and the convex function $\phi$ satisfy \eqref{global-tang-int}--\eqref{global-sep}. 
Let $S(x_{0}, t_{0})$ be a section of $\phi$ with $x_0\in \overline{\Omega}$ and $t_{0}>0$. Then  one of the following is true:
\begin{myindentpar}{1cm}
(i) $S(x_{0}, 2t_{0})$ is an interior section, that is, $S(x_{0}, 2t_{0})\subset \Omega$;\\
(ii) $S(x_{0}, 2 t_{0})$ is included in a boundary section with comparable height, that is, there exists $z\in\partial\Omega$ such that
 $$ S(x_{0}, 2t_{0})\subset S(z, \bar{c}t_{0}).$$
\end{myindentpar}
Here $\bar{c}>1$ is  a constant   depending only on  $\rho,\lambda, \Lambda$ and $n$. 
\end{proposition}
As an application of the dichotomy of sections, we obtain the following volume growth of sections.

\begin{corollary}\label{volume_growth} Assume that the convex domain $\Omega$ and the convex 
function $\phi$ satisfy \eqref{global-tang-int}--\eqref{global-sep}.  Then, 
there exist constants  $c_0, C_{1}, C_{2}$  depending only on $\rho, \lambda, \Lambda$ and $n$ such that for any section $S_{\phi}(x, t)$ with $x\in \overline{\Omega}$  and $t\leq c_0$, we have
\begin{equation}
\label{big_sec_vol}
C_{1} t^{n/2}\leq |S_{\phi}(x, t)|\leq C_{2} t^{n/2}.
\end{equation}
\end{corollary}
By exploiting  the geometric properties of boundary sections, we obtain the following   covering lemma of Besicovitch type.

\begin{lemma}\label{Modified-Besicovitch} Assume that the convex domain $\Omega$ and the convex function $\phi$ satisfy \eqref{global-tang-int}--\eqref{global-sep}.
 Let $A\subset \overline{\Omega}$ and suppose that for each $x\in A\,$ a section $S(x,t)$ is given such that $t$ is bounded by a fixed number $M$. If we denote by $\calF$ the family of all these sections, then there exists a countable 
subfamily of $\calF$, $\{S(x_k,t_k)\}_{k=1}^\infty$, with the following properties:
\begin{itemize}
\item[(i)] $A \subset \bigcup_{k=1}^\infty{S(x_k, t_k)}$;
\item[(ii)] $x_k\not\in \cup_{j<k}{S(x_j, t_j)}\quad \forall k\geq 2$;
\item[(iii)] The family $\{S(x_k,\frac{t_k}{\alpha})\}_{k=1}^\infty$ is disjoint, where $\alpha=2 \theta_{\ast}^2$, and $\theta_{\ast}$ is the engulfing constant in Theorem~\ref{engulfing2};
\item[(iv)] There exists a constant $K>0$ depending only on $\rho, \lambda, \Lambda$ and $n$  such that
\[
\sum_{k=1}^\infty{\chi_{S(x_k, (1-\e)t_k)}(x)}\leq K \log{\frac{1}{\e}}\quad \mbox{for all}\quad 0<\e<1. 
\]
\end{itemize}
\end{lemma}
Our next result is the following covering theorem.
\begin{theorem}\label{thm:covering}
Assume that the convex domain $\Omega$ and the convex function $\phi$ satisfy \eqref{global-tang-int}--\eqref{global-sep}.
Let $\calO\subset\Omega$ open and $\e>0$ small. Suppose that for each $x\in \calO$ a section $S(x,t_x)$ is given with 
\[
\frac{|S(x,t_x) \cap \calO|}{|S(x,t_x)|} = \e.
\]
Then if  $\sup{\{t_x: x\in\calO \}}<+\infty$,  there exists a countable subfamily of sections $\{S(x_k,t_k)\}_{k=1}^\infty$ satisfying
\begin{itemize}
\item[(i)] $\calO \subset \bigcup_{k=1}^\infty{S(x_k, t_k)}$.
\item[(ii)]  $|\calO|\leq \sqrt{\e} \, \big| \bigcup_{k=1}^\infty{S(x_k, t_k)}\big|$.
\end{itemize}
\end{theorem}
As an application of the covering lemma, we have the following global strong-type $p-p$ estimate for the maximal function with respect to sections. 
\begin{theorem}\label{strongtype}
Assume that the convex domain $\Omega$ and the convex function $\phi$ satisfy \eqref{global-tang-int}--\eqref{global-sep}.
For $f\in L^1(\Omega)$,  define
\begin{equation*}\label{maximalfunction}
\mathcal M (f)(x) :=\sup_{t>0}
\dfrac{1}{|S_\phi(x,t)|}\int_{S_\phi(x,t)}|f(y)|\, dy\quad \forall x\in \Omega.
\end{equation*}
Then we have
\begin{enumerate}
\item[(i)]
There exists a constant $C$ depending only on $\rho,  \lambda, \Lambda$ and $n$
such that
\begin{equation*}\label{weaktype1-1}
\big|\{x\in \Omega :\mathcal M (f)(x)>\beta\}\big|\leq
\dfrac{C}{\beta} \, \int_{\Omega}|f(y)|\, dy\quad \forall \beta>0.
\end{equation*}
\item[(ii)]
For any $1<p<\infty$, there exists  $C$ depending only on $p, \rho,  \lambda, \Lambda$ and $n$ such that
\begin{equation*}\label{strongtypep-p}
\left(\int_{\Omega}{|\mathcal M (f)(x)|^p \,dx} \right)^{\frac{1}{p}}\leq
C\, \left(\int_{\Omega}{|f(y)|^p \,dy} \right)^{\frac{1}{p}}.
\end{equation*}
\end{enumerate}
\end{theorem}

Finally, we obtain that the function 
 $d: \overline{\Omega}\times \overline{\Omega}\longrightarrow [0,\infty)$, defined by
\begin{equation*}
d(x,y) := \inf{\big\{r>0: x\in S_\phi(y,r)\mbox{ and } y\in S_\phi(x,r)\big\}} \quad \forall x,y\in\overline{\Omega},
\end{equation*}
is a quasi-distance on $\overline{\Omega}$. Moreover,
$(\overline{\Omega}, d, 
|\cdot|)$ is a space of homogeneous type where $|\cdot|$ denotes the $n$-dimensional Lebesgue measure  restricted to $\overline{\Omega}$. The precise statements of these are given in Section~\ref{sec:quasi-distance}.

\section{Sections of the Monge-Amp\`ere equation and the Localization theorem}
\label{data_sec}
In this section, we recall the main tool to study geometric properties of boundary sections of solutions to the Monge-Amp\`ere equation: the Localization theorem at the boundary for 
solutions to the Monge-Amp\`ere equation (Theorem~\ref{main_loc}). Properties of solutions under suitable rescalings and global regularity for gradient will also be discussed. Throughout this section, we assume that 
the convex domain $\Omega$ and the convex function $\phi$ satisfy \eqref{global-tang-int}--\eqref{global-sep}.
\subsection{The Localization Theorem} We now focus on sections centered at a  point on the boundary $\p\Omega$ and describe their geometry. Assume this boundary point to be $0$ and by
 \eqref{global-tang-int}, we can also assume that
\begin{equation}\label{om_ass}
B_\rho(\rho e_n) \subset \, \Omega \, \subset \{x_n \geq 0\} \cap B_{\frac 1\rho},
\end{equation}
where $\rho>0$ is the constant given by condition \eqref{global-tang-int}. 
After subtracting a linear function, we can assume further that
\begin{equation}\label{0grad}
\phi(0)=0, \quad \nabla \phi(0)=0.
\end{equation}
Let us  denote
$$S(h):= S_{\phi}(0, h).$$
If the boundary data has quadratic growth near the hyperplane $\{x_n=0\}$ then, as $h \rightarrow 0$, $S(h)$ is equivalent to a half-ellipsoid centered at 0. This is the content
of the Localization Theorem proved by Savin in \cite{S1,S2}. Precisely, this theorem reads as follows.

\begin{theorem}[Localization Theorem \cite{S1,S2}]\label{main_loc}
 Assume that $\Omega$ satisfies \eqref{om_ass} and $\phi$ satisfies 
\eqref{eq_u},\eqref{0grad}, and
\begin{equation*}\label{commentstar}\rho |x|^2 \leq \phi(x) \leq \rho^{-1} 
|x|^2 \quad \text{on $\p \Omega \cap \{x_n \leq \rho\}.$}
\end{equation*}
Then, for each $h\leq k$ there exists an ellipsoid $E_h$ of volume $\omega_{n}h^{n/2}$ 
such that
\[
 kE_h \cap \overline \Omega \, \subset \, S(h) \, \subset \, k^{-1}E_h \cap \overline \Omega.
\]
Moreover, the ellipsoid $E_h$ is obtained from the ball of radius $h^{1/2}$ by a
linear transformation $A_h^{-1}$ (sliding along the $x_n=0$ plane)
$$A_hE_h= h^{1/2}B_1,\quad \det A_{h} =1,$$
$$A_h(x) = x - \tau_h x_n, \quad \tau_h = (\tau_1, \tau_2, \ldots, 
\tau_{n-1}, 0), $$
with
$ |\tau_{h}| \leq k^{-1} |\log h|$.
The constant $k$ above depends only on $\rho, \lambda, \Lambda, n$.
\end{theorem}

 The ellipsoid $E_h$, or equivalently the linear map $A_h$, 
provides useful information about the behavior of $\phi$ 
near the origin. From Theorem \ref{main_loc} we also control the shape of sections that are tangent to $\p \Omega$ at the origin. 

\begin{proposition}\label{tan_sec}
Let $\phi$ and $\Omega$ satisfy the hypotheses of the Localization Theorem \ref{main_loc} at the 
origin. Assume that for some $y \in \Omega$ the section $S_{\phi}(y, h) \subset \Omega$
is tangent to $\p \Omega$ at $0$ for some $h \le c$ with $c$ universal, that is, $\p S_{\phi}(y, h)\cap \p\Omega =\{0\}$. Then there exists a small 
 positive constant $k_0<k$ depending 
only on $\rho $, $\lambda$, $\Lambda$  and $n$ such that
$$ \nabla \phi(y)=a e_n 
\quad \mbox{for some} \quad   a \in [k_0 h^{1/2}, k_0^{-1} h^{1/2}],$$
$$k_0 E_h \subset S_{\phi}(y, h) -y\subset k_0^{-1} E_h, \quad \quad k_0 h^{1/2} \le dist(y,\p \Omega) \le k_0^{-1} h^{1/2}, \quad $$
with $E_h$ and $k$ the ellipsoid and constant defined in the Localization Theorem \ref{main_loc}.
\end{proposition}

Proposition \ref{tan_sec} is a consequence of Theorem \ref{main_loc} and was proved in \cite{S3}. \\

The quadratic separation from tangent planes on the boundary for solutions to the Monge-Amp\`ere equation is a crucial assumption in the Localization Theorem~\ref{main_loc}. 
This is the case for solutions to the Monge-Amp\`ere  with the right hand side bounded away from $0$ and $\infty$ and 
smooth boundary data as proved in \cite[Proposition 3.2]{S2}. In particular, the quadratic separation property holds if
$$\phi\mid_{\p\Omega}, \p\Omega\in C^{3},~\text{and}~ \Omega~\text{is uniformly convex}.$$

\subsection{Properties of the rescaled functions}
\label{rescale-sec}

Let $\phi$ and $\Omega$ satisfy the hypotheses of the Localization Theorem \ref{main_loc} at the 
origin. We know that for all $h \le k$,
 $S(h)$ satisfies 
$$k E_h \cap \bar \Omega  \subset S(h) \subset k^{-1} E_h,$$ 
with $A_h$ being a linear transformation 
and
$$\det A_{h} = 1,\quad E_h=A^{-1}_hB_{h^{1/2}},  \quad A_hx=x-\tau_hx_n$$
$$\tau_h \cdot e_n=0, \quad \|A_h^{-1}\|, \,\|A_h\| \le k^{-1} |\log h|.$$
This gives 
\begin{equation}\label{small-sec}
 \overline \Omega \cap B^{+}_{ch^{1/2}/|\log h|}\subset S(h) 
\subset B^{+}_{C h^{1/2} |\log h|}\subset B^{+}_{C h^{1/4} }.
\end{equation}
We denote the rescaled functions by \begin{equation*}
 \phi_h(x):=\frac{\phi(h^{1/2}A^{-1}_hx)}{h}.
\end{equation*}
The function $\phi_h$ is continuous and is defined in $\overline \Omega_h$ with
 $\Omega_h:= h^{-1/2}A_h \Omega,$
 and solves the Monge-Amp\`ere equation
 $$\det D^2 \phi_h=g_h(x), \quad \quad \lambda \le g_h(x) :=g(h^{1/2}A_h^{-1}x) \le \Lambda.$$
 The section at height 1 for $\phi_h$ centered at the origin satisfies
$S_{\phi_{h}}(0, 1)=h^{-1/2}A_hS (h),$ and by the localization theorem we 
obtain $$B_k \cap \overline \Omega_h \subset S_{\phi_{h}}(0, 1) \subset B_{k^{-1}}^+.$$

Some properties of the rescaled function $\phi_h$ was established in \cite[Lemma 4.2]{LS}. For later use, we record them here.
\begin{lemma}
If $h\leq c$, then 

a) for any 
$x,x_{0}\in\p\Omega_{h}\cap B_{2/k}$ we have
\begin{equation*}
 \frac{\rho}{4}\abs{x-x_{0}}^2 \leq \phi_h(x) - \phi_h(x_{0}) -\nabla \phi_h(x_{0}) (x- x_{0})\leq 4\rho^{-1} \abs{x- x_{0}}^2, 
\label{sep-near0}
\end{equation*}

b) if $r \le c$ small, we have $$|\nabla \phi_h| \le 
C r |\log r|^2 \quad \mbox{in} \quad \overline \Omega_h \cap B_r.$$

\label{sep-lem}
\end{lemma}

\subsection{Global regularity}
We note that if $\phi$ satisfies \eqref{eq_u}--\eqref{global-sep}, then for any $\mathbf{v}\in \R^n$ the function $\tilde\phi(x) := \phi(x) + \mathbf{v}\cdot x$ 
also satisfies \eqref{eq_u}--\eqref{global-sep} with the same constants. As a result, under our
hypotheses the gradient of $\phi$ is not bounded by any universal constant. Nevertheless,  the oscillation 
of $\nabla\phi$ is globally bounded thanks to \cite[Proposition~2.6]{LS}. We record 
this result and its direct consequence in the next lemma.
\begin{lemma}\label{global-gradient}
 Assume that the convex domain $\Omega$ and the convex 
function $\phi$ satisfy \eqref{global-tang-int}--\eqref{global-sep}. Then there exist constants $\alpha\in (0,1)$ and $C>0$ depending only on $\rho, \lambda, \Lambda$ and $n$ such that
\[
[\nabla\phi]_{C^{\alpha}(\overline{\Omega})}\leq C.
\]
As a consequence, there is a universal constant $M>0$ satisfying 
\[
 S_\phi(x_0,M)\supset \overline{\Omega} \quad\mbox{for all } x_0\in \overline{\Omega}. 
\]
\end{lemma}
\begin{proof}
 The global $C^{\alpha}$-estimate for the gradient is from \cite[Proposition~2.6]{LS}.  
Now let $x_0\in \overline{\Omega}$. Then for any $x\in \overline{\Omega}$, we have
\[
 \phi(x) -\phi(x_0) -\nabla\phi(x_0)\cdot (x-x_0)\leq [\nabla\phi]_{C^{\alpha}(\overline{\Omega})} \, |x- x_0|^{1+\alpha}\leq C (2\rho^{-1})^{1+\alpha}=:M.
\]
Thus $\overline{\Omega}\subset S_\phi(x_0, M)$ and the lemma follows.
\end{proof}

\section{Geometric properties of sections}
\label{geo_sec}
In this section, we establish several important properties of boundary sections of solutions to the Monge-Amp\`ere equation. Unless otherwise stated, 
the convex domain $\Omega$ and the convex function $\phi$ are assumed to satisfy \eqref{global-tang-int}--\eqref{global-sep}. Under these conditions, we will show that sections of $\phi$ satisfies a dichotomy,  the engulfing and separating properties, and their volumes have expected growth. 
\subsection{Dichotomy of sections and volume growth} 
In this subsection, we shall prove the volume growth of sections and a dichotomy for sections of solutions to the Monge-Am\`pere equation: any section is either an interior section or included in a boundary section with comparable height. 
We begin with the proof of Proposition~\ref{dicho}.

\begin{proof}[Proof of Proposition~\ref{dicho}]
It suffices to consider the case $x_{0}\in\Omega$. Let $S(x_{0}, \bar{h}(x_{0}))$ be the maximal interior section with center $x_{0}$, and let $\bar{h}:= \bar{h}(x_{0})$.  If 
$t_{0}\leq \bar{h}/2$ then $(i)$ is satisfied. We now consider the case $\bh/2 <t_0$ and show that $(ii)$  holds. Without loss of generality, we assume that $\Omega\subset \R^{n}_{+}$, $\p S_{\phi}(x_{0}, \bar{h})$ is tangent to $\p\Omega$ at $0$, and
$\phi(0) =\nabla\phi (0)=0$. It follows that $0= \phi(x_0) - \nabla\phi(x_0) \cdot x_0 +\bh$ yielding
\begin{equation}\label{section-expression}
S_{\phi}(x_{0}, t) =\{x\in\overline{\Omega}: \phi(x)<\nabla\phi(x_0) \cdot x + t-\bh\}\quad \forall t>0.
\end{equation}
Next, we have

{\bf Claim:} There exists a small constant $c>0$ depending only on $\rho, \lambda, \Lambda$ and $n$ such that if
$\bh/2 <t_0\leq c$, then 
\begin{equation}
\label{big_sec_geo}
S_{\phi}(x_{0}, 2 t_{0}) \subset S_{\phi}(0, t^*)\quad\mbox{with}\quad t^*:= k_{0}^{-4} \bar{h} + 2 (2t_0 - \bh).
\end{equation}
Indeed, if $c$ is small, we have, by Proposition \ref{tan_sec}
\begin{equation}\label{gradient-formula}
\nabla \phi (x_{0}) = a e_{n}\quad\mbox{for some}\quad 
a\in [k_{0} \bar{h}^{1/2}, k_{0}^{-1}\bar{h}^{1/2}],
\end{equation}
where $k_{0}>0$ depends only on  $\rho, \lambda, \Lambda$ and $n$. It follows from \eqref{section-expression} and \eqref{gradient-formula} that
\begin{equation*}
S_{\phi}(x_{0}, 2 t_0) =\{x\in\overline{\Omega}: \phi(x)< a x_n  + 2 t_0-\bh\}.
\end{equation*}

Let us choose $c>0$ small enough such that 
$$k_{0}^{-4} \bar{h} + 2 (2t_0 - \bh) \leq 2k_{0}^{-4}c+ 4c \leq k,$$
where $k$ is the constant in the Localization Theorem~\ref{main_loc}. With this choice of $c$, we are going to show that \eqref{big_sec_geo} holds.
Suppose otherwise that \eqref{big_sec_geo} is not true. Then, using the convexity of the sets $S_{\phi}(x_{0}, 2 t_{0})$ 
and $S_{\phi}(0, t^*)$ and the fact that their closures both contain $0$, we can find a point $x\in\Omega$ such that $x\in \p S_{\phi}(0, t^*)\cap S_{\phi}(x_{0}, 2 t_{0})$.
At this point $x$, we have
\begin{align*}
k_{0}^{-4} \bar{h} + 2 (2t_0 - \bh)=t^*=\phi(x)
&< a x_n + 2t_{0} -\bh \leq k_{0}^{-1}\bar{h}^{1/2} x_{n} + 2 t_{0} -\bh.
\end{align*}
Hence,
$$\frac{k_{0}^{-4} \bar{h} + 2t_{0} -\bh}{k_{0}^{-1} \bar{h}^{1/2}}< x_{n}.$$
This together with the Localization Theorem~\ref{main_loc} applied to $S_{\phi}(0, k_{0}^{-4} \bar{h} + 2 (2t_0 - \bh))$ 
gives
$$\frac{k_{0}^{-4} \bar{h} + 2t_{0} -\bh}{k_{0}^{-1} \bar{h}^{1/2}}< k^{-1}\left(k_{0}^{-4} \bar{h} + 2 (2t_0 - \bh)\right)^{1/2}
\leq
k_{0}^{-1}\left(k_{0}^{-4} \bar{h} + 2 (2t_0 - \bh)\right)^{1/2},$$
or 
$$ k_{0}^{-4} \bar{h} + 2 t_0 -\bh < k_{0}^{-2} \bar{h}^{1/2} \left(k_{0}^{-4} \bar{h} + 2 (2t_0 - \bh)\right)^{1/2}.$$
Squaring, we get
$$k_{0}^{-8} \bar{h}^2 + 2  k_{0}^{-4}\bar{h} (2 t_0 -\bh)  + (2 t_0 -\bh)^2 < k_{0}^{-4} \bar{h}\left(k_{0}^{-4} \bar{h} + 2 (2t_0 - \bh)\right)
$$
and as a consequence, $( 2 t_0 -\bh)^2 <0.$
This is a contradiction and hence the claim is proved.

Let $\bar c :=\max{\{2k_{0}^{-4} + 4, M/c\}}$, where $M$ is the universal constant given by Lemma~\ref{global-gradient}. Then
from the claim and since 
$$t^{\ast} \leq (2k_{0}^{-4} + 4) t_{0} \leq \bar{c} t_{0},$$
we see that $(ii)$ holds if $\bh/2 < t_0\leq c$. 

In the case $c<t_0$, by using Lemma~\ref{global-gradient} we obtain
\[
 S_{\phi}(x_{0}, 2 t_{0}) \subset \overline{\Omega} \subset S_{\phi}(0, M)\subset  S_{\phi}(0, \frac{M}{c} t_0)\subset  S_{\phi}(0, \bar{c} t_0)
\]
and thus $(ii)$ also holds true.
\end{proof}

As an application of the dichotomy of sections, we obtain their volume growth.
\begin{proof}[Proof of Corollary \ref{volume_growth}]
Let $c_0:= k/\bar{c}$, where $k$, $\bar{c}$ are the constants in Theorem~\ref{main_loc} and Proposition~\ref{dicho} respectively. Let  $S(x_0,t_0)$ be a section with $t_0\leq c_0$.
If $x_{0}\in \partial\Omega$ then by the Localization Theorem~\ref{main_loc}, we get the desired result. 

Now, we suppose that $x_{0}\in \Omega$ and let $S(x_0, \bh(x_0))$ be the maximal interior section with center $x_0$.  For simplicity, we denote $\bar{h}:= \bar{h}(x_{0})$.
 If $t_{0}\leq \bar{h}/2$, then the result follows from the volume growth of interior sections; see \cite[Corollary 3.2.4]{G}. Therefore, it remains to consider the case  $\bar{h}/2 < t_{0}\leq c_0$. \\
In this case, by Proposition~\ref{dicho} we have $S(x_0, 2 t_0)\subset S(z, \bar{c} t_0)$ for some $z\in \partial\Omega$. Hence, by applying  Theorem~\ref{main_loc} we obtain the second inequality in \eqref{big_sec_vol}. 
To prove the first inequality in \eqref{big_sec_vol}, we first note that, if $t_{0}<2 \bar{h}$ then
$$|S(x_{0}, t_{0})|\geq |S(x_{0}, \frac{\bar{h}}{2})|\geq C_{0}\bar{h}^{n/2}\geq C_{1}t_{0}^{n/2}.$$
Here, the second inequality follows from the volume growth of interior sections; see \cite[Corollary 3.2.4]{G}.
Next, suppose that $t_{0}\geq 2\bar{h}$.
Without loss of generality, we assume that $\Omega\subset \R^{n}_{+}$, $\p S(x_{0}, \bar{h})$ is tangent to $\p\Omega$ at $0$, and $\phi(0) =\nabla\phi (0)=0.$
Then by exactly the same arguments as in the proof of Proposition~\ref{dicho}, we get for some positive number $a$
$$S(x_{0}, t_{0}) =\{x\in\overline{\Omega}: \phi(x)< a x_n  + t_{0}-\bh\}.$$
Using this and the fact  $x_n\geq 0$ in $\Omega$, we can conclude that
$$S(x_{0}, t_{0})\supset \{x\in\overline{\Omega}: \phi(x)< t_{0}-\bar{h}\}\supset \{x\in \overline{\Omega}: \phi(x)<\frac{t_{0}}{2}\}=S_\phi(0, \frac{t_{0}}{2})$$
which together with the Localization Theorem~\ref{main_loc} yields the first inequality in \eqref{big_sec_vol}.
\end{proof}

\subsection{Engulfing and separating properties of sections}

In this subsection, we will establish two important tools: the engulfing and separating properties
of sections, Theorem~\ref{engulfing2} and Proposition~\ref{separating}, respectively. These properties are equivalent.\\

As a first step in the proof of Theorem~\ref{engulfing2}, we prove the engulfing property when the center $x$ lies on the boundary, that is, $S_\phi(x, t)$ is a boundary section. 
Without loss of generality, we assume that $x$ is the origin and we write $S(t)$ for $S_\phi(0, t)$. Furthermore, we assume that $\phi$ and $\Omega$ satisfy the hypotheses of the Localization Theorem \ref{main_loc} at the 
origin.
\begin{lemma}\label{engulfing}
There exists  $\theta_{0}>0$ depending only on $\rho$, $\lambda$, $\Lambda$ and $n$ such that 
if $X\in S(t)$ with $t>0$, then 
$$S(t)\subset S_{\phi}(X, \theta_{0} t).$$
\end{lemma}
\begin{proof} Let $t\leq c_{0}$ with $c_{0}\leq k$ to be chosen later, where $k=k(\rho,\lambda,\Lambda,n)$ is the small constant in the Localization Theorem~\ref{main_loc}. Let us consider $h\in [t, k]$.  Let $A_{h}$ be the linear transformation associated with the section $S(h)$ as determined by the above theorem. 
Let 
\[\phi_{h}(z) :=\frac{\phi(h^{1/2}A_{h}^{-1}z)}{h}\quad\mbox{for}\quad
z\in\Omega_{h} := h^{-1/2} A_{h}\Omega.
 \]
For $X, Y\in S(t)$, we define
$$x := h^{-1/2} A_{h} X, \quad y := h^{-1/2}A_{h}Y.$$
Then
$$S_{\phi_{h}}(0, 1) = h^{-1/2} A_{h} S(h); \quad  x, y \in S_{\phi_{h}}(0, \frac{t}{h})$$
and furthermore,
$$\frac{1}{h} \big[\phi(Y)-\phi(X)- \nabla \phi(X)\cdot (Y-X)\big]= \phi_{h}(y)- \phi_{h}(x)-\nabla \phi_{h}(x)\cdot (y-x).$$
By Lemma \ref{sep-lem} (a), $\phi_{h}$ also satisfies the hypotheses of the Localization Theorem~\ref{main_loc} in $S_{\phi_{h}}(0, 1)$. Hence, by \eqref{small-sec}, we have for $C$ universal
$$|x|, |y|\leq C(\frac{t}{h})^{1/2}|log (\frac{t}{h})|.$$
Now, we take $h>0$ satisfying $t/h= 1/M_{1}$ with $M_{1}>1$ is chosen so that 
$$C(\frac{t}{h})^{1/2}|log (\frac{t}{h})| = CM_{1}^{-1/2} log M_{1}\leq c,$$
where $c$ is determined by Lemma~\ref{sep-lem}. Given this choice of $M_{1}$, it suffices to take $c_{0} := \frac{k}{M_{1}}$ so as to have $h\in [t, k]$. Then, by Lemma~\ref{sep-lem} (b), 
$$|\nabla \phi_{h}(x)|\leq c |log c|^{2}.$$
Thus,
$$\frac{1}{h} \big[\phi(Y)-\phi(X)- \nabla \phi(X)\cdot(Y-X)\big] = \phi_{h}(y)- \phi_{h}(x)-\nabla \phi_{h}(x)\cdot (y-x) \leq \frac{t}{h} + 2 c^2 |log c|^2$$
implying
$$Y\in S_{\phi}\big(X, (1 + 2 M_1 c^2 |log c|^2) t\big).$$
Hence for any $X\in S(t)$ with $t\leq c_0$, we get
$$S(t)\subset S_{\phi}\big(X, (1 + 2 M_1 c^2 |log c|^2) t\big).$$
In the case $X\in S(t)$ with $t>c_0$, then by using Lemma~\ref{global-gradient} we obtain
\begin{align*}
 S(t)\subset \overline{\Omega}\subset S_\phi(X, M) \subset S_\phi(X, \frac{M}{c_0} t). 
\end{align*}
Therefore, by taking $\theta_{0} := \max{\{1 + 2 M_1 c^2 |log c|^2, M/c_0\}}$, we see that 
$S(t)\subset S_{\phi}(X, \theta_{0} t)$  for any $t>0$.
\end{proof}

We are now ready to give the proof of Theorem \ref{engulfing2}.
\begin{proof}[Proof of Theorem \ref{engulfing2}]
By  Lemma \ref{engulfing}, it remains to consider the case  $x\in \Omega$. Let $S(x,\bh(x))$ be the  maximal interior section with center $x$.\\
 If $t< \bar{h}(x)/2$ then $S(x, 2t)$ is an interior section and the result follows from the engulfing properties of
interior sections of the Monge-Amp\`ere equation with bounded right hand side 
(see the proof of Theorem~3.3.7 in \cite{G}), namely, 
$$S(x, t)\subset S(y, \theta t)$$
for some $\theta$ depending only on $\lambda, \Lambda$ and $n$.\\
Now, consider the case 
$\bar{h}(x)/2\leq t$. Then  Proposition~\ref{dicho} implies that 
\begin{equation*}
 S(x, 2 t) \subset S(z, \bar{c}t)
\end{equation*}
for some $z\in\partial\Omega$. Since $y\in S(z, \bar{c}t)$, by the engulfing property of boundary sections  from Lemma~\ref{engulfing}, we have 
$$S(z, \bar{c} t)\subset S(y, \theta_{0} \bar{c} t).$$
Therefore the result follows with $\theta_{\ast} :=\max\{\theta, \theta_{0} \bar{c}\}$ noting that
$S(x, t) \subset S(z, \bar{c}t)$.
\end{proof}

Finally, we prove the separating property of sections. 

\begin{proof}[Proof of Proposition \ref{separating}] 
(i) Suppose that $y\not\in S(x, t)$.
If $z\in S(y, \frac{t}{\theta_{\ast}^2})\cap S(x, \frac{t}{\theta_{\ast}^2})$, then, by Theorem \ref{engulfing2}, 
$$S(y, \frac{t}{\theta_{\ast}^2})\cup S(x, \frac{t}{\theta_{\ast}^2})\subset S(z,\frac{t}{\theta_{\ast}}).$$
This implies that
$y, x\in S(z,\frac{t}{\theta_{\ast}})$ and also by the engulfing property that
$S(z,\frac{t}{\theta_{\ast}})\subset S(x, t)$. Therefore, $y\in S(x, t)$ which is a contradiction and so 
$S(y, \frac{t}{\theta_{\ast}^2})\cap S(x, \frac{t}{\theta_{\ast}^2})=\emptyset$.\\
(ii) Suppose $y\in S(x,t)$.  We need to prove that
\begin{equation}
\label{engulfing3}
S(x, t)\subset S(y, \theta^2 t).
\end{equation}
Suppose that it is not true. Then, we can find $z\in S(x, t)$ such that
\begin{equation}
\label{out-section}
z\not\in S(y, \theta^2 t).
\end{equation}
Since $y, z\in S(x, t)$, it follows from the separating property that
\begin{equation}
\label{in-section}
x\in S(y,\theta t)\cap S(z,\theta t).
\end{equation}
Applying the separating property to \eqref{out-section}, we get
$S(z, \theta t)\cap S(y, \theta t) =\emptyset.$
This is a contradiction to \eqref{in-section}. Hence (\ref{engulfing3}) holds as desired.
\end{proof}

\section{A covering theorem and maximal function estimates}
\label{covering_sec}

In this section we establish  a covering lemma of Besicovitch type. Using this lemma, we prove a covering theorem for boundary sections and derive a global strong-type $p-p$ estimate for the maximal function. Our results extend those in \cite{CG1}  
where interior sections are considered. We begin with the proof of Lemma \ref{Modified-Besicovitch}.

\begin{proof}[Proof of Lemma \ref{Modified-Besicovitch}]
We may assume  that $M= \sup\{t:  S(x,t)\in \calF\}$. Let us first consider the family
\[
\calF_0 := \{S(x,t)\in \calF:\, \frac{M}{2}< t\leq M\} \quad \mbox{and}\quad A_0 := \{x: \, S(x,t)\in \calF_0\}. 
\]
Pick $S(x_1,t_1)\in \calF_0$ such that $t_1> 3M/4$. If  $A_0\subset S(x_1,t_1)$, we stop. Otherwise, the set
\[\Big\{t:\, S(x,t)\in\calF_0\mbox{ and } x\in A_0\setminus S(x_1,t_1)\Big\}
\]
is nonempty and we let $\alpha_2$ denote its supremum. Pick  $t_2$ in this set such that $t_2 > 3\alpha_2/4$, and let $S(x_2, t_2)$ be the corresponding section. We then have 
$x_2\not\in S(x_1, t_1)$ and $t_1> 3M/4 \geq 3\alpha_2/4\geq  3 t_2/4$.
Again, if $A_0\subset S(x_1,t_1)\cup S(x_2, t_2)$ we stop. Otherwise, we continue the process. As a result, we have constructed a family, possible infinite, which we denote by
\[
\calF_0' =\{S(x^0_i, t^0_i)\}_{i=1}^\infty \quad \mbox{with}\quad x^0_j\in A_0\setminus \bigcup_{i<j}{S(x^0_i, t^0_i)}.
\]  

We next consider the family
\[
\calF_1 := \{S(x,t)\in \calF:\, \frac{M}{4}< t\leq \frac{M}{2}\} \quad \mbox{and}\quad A_1 := \Big\{x: \, S(x,t)\in \calF_1\mbox{ and } x\not\in \bigcup_{i=1}^\infty{S(x^0_i, t^0_i)}\Big\}. 
\]
We repeat the above construction for the set $A_1$ and obtain a family of sections denoted by 
\[
\calF_1' =\{S(x^1_i, t^1_i)\}_{i=1}^\infty \quad \mbox{with}\quad x^1_j\in A_1\setminus \bigcup_{i<j}{S(x^1_i, t^1_i)}.
\]  

We continue this process and in the $k$th-stage we consider the family
\[
\calF_k := \{S(x,t)\in \calF:\, 
\frac{M}{2^{k+1}}< t\leq \frac{M}{2^k}\}
\]
and 
\[
A_k := \Big\{x: \, S(x,t)\in \calF_k
\mbox{ and }  x\not\in \mbox{sections previously selected}\Big\}. 
\]
In the same way as before, we obtain a family of sections denoted by 
\[
\calF_k' =\{S(x^k_i, t^k_i)\}_{i=1}^\infty \quad \mbox{with}\quad x^k_j\in A_k\setminus \bigcup_{i<j}{S(x^k_i, t^k_i)}.
\]  

We are going to show that the collection of all sections in all generations $\calF_k'$, $k\geq 0$, is the family that satisfies the conclusion of the lemma.

{\bf Claim 1:} The overlapping in each generation $\calF_k'$ is at most $\kappa$, where $\kappa$ depends only on $\rho, \lambda, \Lambda$ and $n$.
To show this, let us suppose that
\[
 z\in S(x^k_{j_1}, t^k_{j_1})\cap \cdots \cap S(x^k_{j_N}, t^k_{j_N}),
\]
with $S(x^k_{j_i}, t^k_{j_i})\in \calF_k'$ and $j_1< j_1<\cdots <j_N$. For simplicity we 
set $x^k_{j_i}=x_i$, $t^k_{j_i}=t_i$, and let $t_0$ be the maximum of all these $t_i$. 
Then, by the engulfing property of sections in Theorem~\ref{engulfing2}, we have 
\begin{equation}\bigcup_{i=1}^{N} S(x_{i}, t_{i})\subset S(z, \theta_{\ast}t_{0}).
 \label{all_included}
\end{equation}
For any $l>i$, as $x_l\not\in S(x_i, t_i)$ we obtain  from the separating property of sections in Proposition~\ref{separating} that  $$S(x_{l}, \frac{t_{i}}{\theta_{\ast}^2})\cap S(x_{i}, \frac{t_{i}}{\theta_{\ast}^2})=\emptyset.$$ 
Since  
$$\frac{M}{2^{k+1}}<t_{i}\leq \frac{M}{ 2^{k}},$$
we then conclude 
that $t_{i}>t_{0}/2$ and thus
\begin{equation}\label{disjoint-same-generation}
S(x_{l}, \frac{t_{l}}{2\theta_{\ast}^2})\cap S(x_{i}, \frac{t_{i}}{2 \theta_{\ast}^2})=\emptyset;\quad
S(x_{l}, \frac{t_{0}}{2\theta_{\ast}^2})\cap S(x_{i}, \frac{t_{0}}{2 \theta_{\ast}^2})=\emptyset.
\end{equation}
Combining this fact with \eqref{all_included}, we obtain
\begin{equation}\label{eq:claim1}
\sum_{i=1}^{N} |S(x_{i},\frac{t_{0}}{2 \theta_{\ast}^2})|\leq |S(z, \theta_{\ast}t_{0})|.
\end{equation}
Let $c_0$ be the universal constant in Corollary~\ref{volume_growth}. If $\theta_{\ast} t_0\leq c_0$, then 
it follows from \eqref{eq:claim1} and the volume estimates \eqref{big_sec_vol} that
$$\sum_{i=1}^{N} C_{1} \left(\frac{t_{0}}{2 \theta_{\ast}^2}\right)^{n/2} \leq C_{2} (\theta_{\ast} t_{0})^{n/2}$$
and thus
 $$N\leq N_1 :=\frac{C_{2}}{C_{1}} 2^{n/2} \theta_{\ast}^{3n/2}.$$
In the case $\theta_{\ast} t_0 > c_0$, then
\eqref{eq:claim1} implies  $\sum_{i=1}^{N} |S(x_{i},\frac{c_0}{2 \theta_{\ast}^3 })|\leq |\Omega|\leq |B_{1/\rho}|$. So by applying  \eqref{big_sec_vol} again, we get
\[
\sum_{i=1}^{N}{C_1 \left(\frac{c_0}{2 \theta_{\ast}^3}\right)^{n/2}} \leq |B_{1/\rho}|
\]
 yielding
\[
N\leq N_2 := \frac{|B_{1/\rho}|}{C_1 c_0^{n/2}}  2^{n/2} \theta_{\ast}^{3n/2}.
\]
Therefore the overlapping in each generation $\calF_k'$ is at most $\kappa$, where $\kappa :=\max{\{N_1, N_2\}}$.

{\bf Claim 2:} The family $\calF_k' =\{S(x^k_i, t^k_i)\}_{i=1}^\infty$ is actually finite. Indeed, by Claim~1
\[
\sum_{i}{\chi_{S(x^k_i, t^k_i)}(x)}\leq \kappa
\]
and hence by integrating over $\Omega$ we obtain
\[
\sum_{i}{|S(x^k_i, t^k_i)|}\leq  \kappa \, |\Omega|.
\]
Note that $M/2^{k+1} < t^k_i$. Therefore if we let $a :=\min{\{M/2^{k+1}, c_0\} }$, then 
it follows from the above inequality and Corollary~\ref{volume_growth}  that
\[
\sum_{i}{C_1 a^{n/2}}\leq \sum_{i}{|S(x^k_i, a)|} \leq \kappa\, |\Omega|,
\]
implying that the number of terms in the sum is finite and Claim~2 is proved.

 From Claim~2 and our construction we get $A_k\subset \cup_{i=1}^\infty{S(x^k_i, t^k_i)}$ and thus $(i)$ holds. Also since each generation $\calF_k'$ has 
a finite number of members, by relabeling the indices of all members of all generations $\calF_k'$ we obtain $(ii)$.

In order to prove property $(iii)$, let $x_i \neq x_j$.  If $S(x_i, t_i)$ and $S(x_j, t_j)$ belong to the same generation, then $\, S(x_i, t_i/\alpha)\cap S(x_j, t_j/\alpha) =\emptyset\,$ by \eqref{disjoint-same-generation}, 
where $\alpha := 2\theta_{\ast}^2$. On the other hand, suppose $S(x_i, t_i) \in \calF_k'$ and $S(x_j, t_j)\in \calF_{k+p}'$ for some $p\geq 1$. Then, by construction, $x_j\not\in S(x_i, t_i)$ and so
$S(x_j, t_i/\theta_{\ast}^2)\cap S(x_i, t_i/\theta_{\ast}^2) =\emptyset$ by the separating property in Proposition~\ref{separating}. Since $t_i>t_j$, this gives  
$S(x_j, t_j/\theta_{\ast}^2)\cap S(x_i, t_i/\theta_{\ast}^2) =\emptyset$ and so $(iii)$ is proved.

{\bf Claim 3:} Assume  $0<r \leq c_0$. Then the number of sections $S(x,t)$  in the family $\{S(x_k, t_k)\}$  with $t\geq r$  is bounded by a constant $N$ depending only on $r, \rho, \lambda, \Lambda$ and $n$. To see this, let us denote this subfamily of sections by $\{S(x_k,t_k)\}_{k\in I}$. Then by using property $(iii)$ and Corollary~\ref{volume_growth} we obtain
\begin{align*}
|B_{\frac{1}{\rho}}|\geq |\Omega|\geq \big|\bigcup_{k\in I} S(x_k, \frac{t_k}{\alpha})\big| =\sum_{k\in I} \big| S(x_k, \frac{t_k}{\alpha})\big|
\geq \sum_{k\in I} \big| S(x_k, \frac{r}{\alpha})\big| \geq \sum_{k\in I} C_1 \big(\frac{r}{\alpha}\big)^{n/2}.
\end{align*}
Thus the  number of elements in $I$ is bounded by $N:=C_1^{-1} (\alpha r^{-1})^{n/2} |B_{1/\rho}|$.

  We next estimate the overlapping of sections belonging to different generations. Let $0<\e<1$ and
\begin{equation}\label{overlapping}
z\in \bigcap_{i} S\big(x^{e_i}_{r_i}, (1-\e) t^{e_i}_{r_i}\big),
\end{equation}
where $e_1< e_2<\cdots$, $M 2^{-(e_i +1)} < t^{e_i}_{r_i} \leq M 2^{- e_i}$, and for simplicity in the notation  we set $x_i= x^{e_i}_{r_i}$ and $t_i = t^{e_i}_{r_i}$.
Our aim is to show that the number of sections in \eqref{overlapping} is not more than $C \log{\frac{1}{\e}}$.
We only need to consider $\e< 1-\frac{1}{\alpha}$ since otherwise the sections are disjoint by $(iii)$. Let 
$r_0 := \min{\{\frac{k}{\bar{c} \theta_{\ast} M_1}, \frac{c}{\bar{c} \theta_{\ast} M_1}, c_0\}}$, where $c$ is the constant in Lemma~\ref{sep-lem} and  $M_{1}$ 
is chosen as in the proof of Lemma~\ref{engulfing}.  Then by Claim~3 and in view of our purpose, we can assume 
without loss of generality that $t_i\leq r_0$ for all $i$ appearing in \eqref{overlapping}. Now for any $j>i$, we claim that 
\begin{equation}\label{gap-estimate}
e_j - e_i \leq C \log{\frac{1}{\e}},
\end{equation}
where $C>0$ depends only on $\rho, \lambda, \Lambda$ and $n$.
In particular, the number of members in \eqref{overlapping} is at most $C \log{\frac{1}{\e}}$, which together with 
Claim~1 gives $(iv)$ as desired. 

To prove the claim,  observe first that  by the engulfing property of sections
in Theorem~\ref{engulfing2} we have
\begin{equation}
\label{enlarge-sec}
S(x_j, t_j) \subset S(z, \theta_{\ast} t_j)\quad\mbox{and}\quad
S(x_i, t_i) \subset S(z, \theta_{\ast} t_i).
\end{equation}
Let $S(z,\bh(z))$ be the maximal interior section with center $z$ when $z$ is an interior point of $\Omega$. The case $z\in \partial\Omega$ will
be dealt with briefly at the end of the proof. We then consider the following possibilities:

{\bf Case 1:} $\theta_{\ast} t_i< \bh(z)$. Then both sections $S(x_i, t_i)$ and $S(x_j, t_j)$ are interior sections and \eqref{gap-estimate} follows from the proof in \cite[Lemma~6.5.2]{G}. We include 
the proof here for the sake of completeness.  Let $T$ be an affine map normalizing the section $S(x_i, t_i)$. Then 
since $t_i>t_j$, by \cite[Theorem~3.3.8]{G}, there exists $\e_{1}$ depending on $\rho, \lambda, \Lambda$ and $n$  such that
\[
T (S(x_j,  t_j))\subset  B\Big(T x_j, K_1\big(\frac{t_j}{t_i}\big)^{\e_1}\Big).
\]
 By construction $x_j \not\in S(x_i, t_i)$ and hence by \cite[Corollary~3.3.6]{G} we obtain
\[
 B(T x_j, C \e^n) \cap T(S(x_i, (1-\e)t_i))=\emptyset.
\]
We deduce from the above two relations that
\[
 C \e^n < |T x_j- T z| \leq K_1\big(\frac{t_j}{t_i}\big)^{\e_1} \leq K_1 2^{\e_1} 2^{(e_i - e_j)\e_1}
\]
implying \eqref{gap-estimate}.

{\bf Case 2:} $\theta_{\ast} t_i\geq \bh(z)$ and $\theta_{\ast} t_j\leq  \bh(z)/2$. We can assume that $\Omega\subset \R^{n}_{+}$, $\p S(z, \bar{h}(z))$ is tangent to $\p\Omega$ at $0$, and
$\phi(0) =\nabla\phi (0)=0$.
Then, by Proposition \ref{dicho}, we know that 
\begin{equation*}
S(z,\theta_{\ast} t_i)\subset  S(0, \bar{c} \theta_{\ast} t_i).
\end{equation*}
Let $M_{1}$ be chosen as in the proof of Lemma \ref{engulfing} and denote
$$C_{0} : = \bar{c} \theta_{\ast} M_{1};\quad h : = C_{0} t_{i}.$$
Then $h\leq k$ where $k$ is the constant in the Localization Theorem \ref{main_loc}.
We use the notation as in Subsection \ref{rescale-sec}. We rescale $\phi$ by \begin{equation*}
 \phi_h(x):=\frac{\phi(h^{1/2}A^{-1}_hx)}{h}
\end{equation*}
where $A_{h}$ is the linear map in the Localization Theorem~\ref{main_loc}.
The function $\phi_h$ is continuous and is defined in $\overline \Omega_h$ with
 $\Omega_h:= h^{-1/2}A_h \Omega,$
 and solves the Monge-Amp\`ere equation
 $$\det D^2 \phi_h=g_h(x), \quad \quad \lambda \le g_h(x) :=g(h^{1/2}A_h^{-1}x) \le \Lambda.$$
 The section at height 1 for $\phi_h$ centered at the origin satisfies
$S_{\phi_{h}}(0, 1)=h^{-1/2}A_hS (h),$ and by the localization theorem we obtain $$B_k \cap \overline \Omega_h \subset S_{\phi_{h}}(0, 1) \subset B_{k^{-1}}^+.$$
Let $T:= h^{-1/2} A_{h}.$
Due to our assumptions  $$\theta_{\ast} t_j \leq \bh(z)/2< C_0 t_i = h,$$ the 
section $S_{\phi_{h}}(T z, \theta_{\ast} t_j / h) = T(S(z, \theta_{\ast} t_j))$ is an 
interior section of $\phi_{h}$ in $S_{\phi_{h}}(0, 1)$. Therefore, from the proof of \cite[Theorem 3.3.8]{G}, we find 
$\e_{1}$ depending on $\rho,\lambda, \Lambda$ and $n$ such that
\[
S_{\phi_{h}}(T z, \frac{\theta_{\ast} t_j}{C_0 t_i})  \subset B(T z, K \big(\frac{\theta_{\ast} t_j}{ C_0 t_i}\big)^{\e_1}).
\]
Thus, by recalling \eqref{enlarge-sec}, we obtain
\begin{equation}\label{controlled-distance}
|T x_j - Tz|\leq K_1  \big(\frac{t_j}{t_i}\big)^{\e_1}.
\end{equation}
 We have
$$Tx_{i}, Tx_{j}, Tz\in T (S(z, \theta_{\ast}t_{i}))\subset T (S(0, \bar{c}\theta_{\ast} t_{i}))= S_{\phi_{h}}(0, \frac{\bar{c}\theta_{\ast} t_{i}}{C_{0}t_{i}})=
S_{\phi_{h}}(0, \frac{1}{M_{1}}).$$
By Lemma~\ref{sep-lem} (a), $\phi_{h}$ also satisfies the hypotheses of the Localization Theorem~\ref{main_loc} in $S_{\phi_{h}}(0, 1)$. Hence, by \eqref{small-sec}, $Tx_{i}, Tx_{j}, Tz$ belong to $B(0, c)$. As a result, we obtain from 
Lemma~\ref{sep-lem} (b) that, 
\begin{equation}\label{renormalized-gradient-est}
|\nabla \phi_{h}(Tx_{i})|, |\nabla \phi_{h}(w)| \leq c|log c|^2 
\end{equation}
for any $w\in [Tx_{j}, Tz]$.
Since $T x_j \not\in T(S(x_i, t_i))=S_{\phi_{h}}( T x_i, \frac{1}{C_0})$, we have
\begin{equation*}
\frac{1}{C_0} \leq  \phi_{h}(Tx_{j})- \phi_{h}(Tx_{i}) - \nabla \phi_{h}(Tx_{i})\cdot (Tx_{j}-Tx_{i}).
\end{equation*}
By rewriting the above right hand side in the form
$$\phi_{h}(Tz)- \phi_{h}(Tx_{i}) - \nabla \phi_{h}(Tx_{i})\cdot (Tz-Tx_{i}) + \phi_{h}(T x_j)- \phi_{h}(T z)-\nabla \phi_{h}(T x_i)\cdot (T x_j - Tz)$$
and using the fact that $T z \in T(S(x_i, (1-\e) t_i))=S_{\phi_{h}}( T x_i,  \frac{1-\e}{C_0})$, we  obtain
\begin{align} \label{epsilon-estimate}
\frac{1}{C_0} 
\leq  \frac{1-\e}{C_0}
+ \phi_{h}(T x_j)- \phi_{h}(T z)-\nabla \phi_{h}(T x_i)\cdot (T x_j - Tz)
\leq \frac{1-\e}{C_0}
+ C |T x_j- T z|
\end{align}
where \eqref{renormalized-gradient-est} is used to obtain the last inequality.
It follows from \eqref{controlled-distance} and \eqref{epsilon-estimate} that
\begin{align*}
\e  
\leq  C |T x_j- T z| \leq C \big(\frac{t_j}{t_i}\big)^{\e_1}\leq  C 2^{(e_i - e_j)\e_1}
\end{align*}
giving \eqref{gap-estimate}.

{\bf Case 3:} $\theta_{\ast} t_i\geq  \bh(z)$ and $\theta_{\ast} t_j>  \bh(z)/2$. We can assume 
that $\Omega\subset \R^{n}_{+}$, $\p S(z, \bar{h}(z))$ is tangent to $\p\Omega$ at $0$, and
$\phi(0) =\nabla\phi (0)=0$.
Then Proposition~\ref{dicho} gives
\begin{equation*}
S(z,\theta_{\ast} t_j)\subset  S(0, \bar{c}\theta_{\ast} t_j)
\quad\mbox{and} \quad S(z,\theta_{\ast} t_i)\subset  S(0, \bar{c}\theta_{\ast} t_i).
\end{equation*}
Let  $C_0$, $h$, $\phi_{h}$ and $T$ be defined as  in Case~2. 
By Lemma~\ref{sep-lem} (a), $\phi_{h}$ also satisfies the hypotheses of the Localization Theorem \ref{main_loc} in $S_{\phi_{h}}(0, 1)$. By this theorem and \eqref{small-sec}, we have
\[
 T(S(z,\theta_{\ast} t_j))\subset T(S(0, C_0 t_j))
= S_{\phi_{h}}(0, \frac{t_j}{t_i})  \subset  B\Big(0 , K\big(\frac{t_j}{t_i}\big)^{\frac14}\Big),
\]
as we can assume $t_j/t_i <k$ to prove \eqref{gap-estimate}.
It follows that
\begin{equation*}
|T x_j - Tz|\leq 2 K  \big(\frac{t_j}{t_i}\big)^{\frac14}.
\end{equation*}
We infer from  this and the same estimate as \eqref{epsilon-estimate} that
\begin{align*}
\e  
\leq  C \big(\frac{t_j}{t_i}\big)^{\frac14}\leq  C 2^{\frac14 (e_i - e_j)}
\end{align*}
giving \eqref{gap-estimate}. \\

Finally, we remark that when $z$ in \eqref{overlapping} is a boundary point of $\Omega$, say $0\in\p\Omega$ where  $\Omega\subset \R^{n}_{+}$,
then we also obtain \eqref{gap-estimate} exactly as in {\bf Case 3} of the interior points. 
Thus the proof of the lemma is complete.
\end{proof}

With the help of Lemma \ref{Modified-Besicovitch}, we are able to give the proof of the covering theorem.

\begin{proof}[Proof of Theorem \ref{thm:covering}]
 Let $0<\mu<1/2$ be arbitrary. By applying 
Lemma~\ref{Modified-Besicovitch} to the family $\calF :=\{S(x, t_x)\}_{x\in \calO}$, there exists a countable subfamily, denoted by $\{S(x_k, t_k)\}_{k=1}^\infty$, such that
$
\calO\subset  \bigcup_{k=1}^\infty{S(x_k, t_k)}$ and 
\begin{equation}\label{finite-overlap} 
\sum_{k=1}^\infty{\chi_{S(x_k, (1-\mu)t_k)}(x)}\leq K \log{\frac{1}{\mu}}.
\end{equation}
Let us write $S_k$ for $S(x_k, t_k)$ and $S_k^\mu$ for $S(x_k, (1-\mu)t_k)$. 
Then we have  
\begin{equation*}
|\calO| =|\calO \cap \cup_{k=1}^\infty S_k| = \lim_{N\to\infty}{|\calO \cap  \cup_{k=1}^N  S_k|}
\leq \limsup_{N\to \infty}{\sum_{k=1}^N |\calO \cap S_k|} 
= \e \limsup_{N\to \infty}{\sum_{k=1}^N |S_k|}. 
\end{equation*}
Moreover, by the doubling property  in 
Lemma~\ref{doubling-engulfing} (ii), we get
\[
|S_k| \leq C \, |S(x_k, \frac{t_k}{2})|
\leq  C \, |S\big(x_k, (1-\mu)t_k\big)| =C \, |S_k^\mu|.
\]
Therefore,
\begin{equation}\label{estimating-O}
|\calO| \leq C \e \limsup_{N\to \infty}{\sum_{k=1}^N |S_k^\mu|}.
\end{equation}
Next let $n_N^\mu(x)$ be the overlapping function for the family $\{S_k^\mu\}_{k=1}^N$ as in 
the proof of   \cite[Theorem~6.3.3]{G}, that is, 
\begin{equation*}
n^{\mu}_{N}(x) :=
\begin{cases}
\# \{k: x\in S^{\mu}_{k}\} & \text{if } x\in \cup_{k=1}^{N} S^{\mu}_{k},
\\
1 & \text{if }  x\not\in \cup_{k=1}^{N} S^{\mu}_{k}.

\end{cases}
\end{equation*}
Then  
$$\chi_{ \cup_{k=1}^{N} S^{\mu}_{k}}(x) = \frac{1}{n^{\mu}_{N}(x)}\sum_{k=1}^{N} \chi_{S^{\mu}_{k}}(x)$$
and
$n_N^\mu(x)\leq K \log{\frac{1}{\mu}}$ by \eqref{finite-overlap}, and hence 

\begin{align*}
  \sum_{k=1}^N{|S_k^\mu|}
  &= \int_{\Omega}{n_N^\mu(x) \frac{1}{n_N^\mu(x)} \sum_{k=1}^N{\chi_{S_k^\mu}}(x)\,dx}\leq  K \log{\frac{1}{\mu}} \int_{\Omega}{\frac{1}{n_N^\mu(x)} \sum_{k=1}^N{\chi_{S_k^\mu}}(x)\,dx}\\
  &=K \log{\frac{1}{\mu}} \int_{\Omega}{\chi_{\cup_{k=1}^N S_k^\mu}(x)\,dx}  =  K \log{\frac{1}{\mu}} \, |\cup_{k=1}^N{S_k^\mu}|.
\end{align*}
We infer from this and \eqref{estimating-O} that
\begin{align*}
|\calO| \leq C K \e \log{\frac{1}{\mu}} \, |\cup_{k=1}^\infty{S_k^\mu}|   \quad\mbox{for all}\quad 0<\mu <1/2.
\end{align*}
By choosing $\mu>0$ such that $\log{\frac{1}{\mu}}= 1/ (C K \sqrt{\e})$, we obtain $(ii)$  as desired.
\end{proof}

We end this section by establishing   some global estimates for the maximal function with respect to sections. The proof is based on  the covering lemma (Lemma~\ref{Modified-Besicovitch}) and the standard method.

\begin{proof}[Proof of Theorem \ref{strongtype}]
Let $A_\beta :=\{x\in \Omega : \calM(f)(x)>\beta \}$ and $M$ be the constant in Lemma~\ref{global-gradient}. By Lemma~\ref{global-gradient}, we have
\[
\dfrac{1}{|S_\phi(x,t)|}\int_{S_\phi(x,t)}|f(y)|\, dy =\dfrac{1}{|\Omega|}\int_{\Omega}|f(y)|\, dy\quad \forall t\geq M
\]
which implies that
\begin{equation*}
\mathcal M (f)(x)=\sup_{t\leq M}
\dfrac{1}{|S_\phi(x,t)|}\int_{S_\phi(x,t)}|f(y)|\, dy\quad \forall x\in \Omega.
\end{equation*}
Therefore for each $x\in A_\beta$, we can find $t_x \leq M$ satisfying
\begin{equation*}
\frac{1}{|S_\phi (x,t_x)|}\int_{S_\phi (x,t_x)} |f(y)|\, dy
\geq \beta.
\end{equation*}
Consider the family $\{S_\phi (x,2t_x)\}$. Then by Lemma~\ref{Modified-Besicovitch}, there exists a countable subfamily
$\{S_\phi (x_k,2t_k)\}_k$ such that $A_\beta \subset \bigcup_{k}
S_\phi (x_k,2t_k)$ and  $\sum_k \chi_{s_\phi
(x_k,(1-\epsilon)2t_k)}(x)\leq C\,\log\frac{1}{\epsilon}$ for
every $0<\epsilon <1/2$. In particular, 
\begin{equation*} |A_\beta| \leq \sum_k |S_\phi
(x_k,2t_k)|
\leq C \sum_k |S_\phi
(x_k,t_k)|
\leq C\sum_k |S_\phi (x_k,(1-\epsilon)2t_k)|
\end{equation*}
noting that  $|S_\phi
(x_k,2t_k)| \leq C\, |S_\phi
(x_k,t_k)|$ by the doubling property in Lemma~\ref{doubling-engulfing} (ii).
But  as
\begin{align*}
\beta \leq
\frac{1}{|S_\phi(x_k,t_k)|}\int_{S_\phi(x_k,t_k)}|f(y)|
\,dy\leq \frac{C}{|S_\phi
(x_k,(1-\epsilon)2t_k)|}\int_{S_\phi(x_k,(1-\epsilon)2t_k)}
|f(y)|\,dy,
\end{align*}
we conclude that
\begin{align*}
|A_\beta| & \leq 
\frac{C}{\beta}\sum_k
\int_{S_\phi(x_k,(1-\epsilon)2t_k)}|f(y)|\,dy
=\frac{C}{\beta}\sum_k \int_\Omega
\chi_{S_\phi(x_k,(1-\epsilon)2t_k)}(y)\,|f(y)|\,dy\\
&=\frac{C}{\beta}\int_\Omega
\sum_k
\chi_{S_\phi(x_k,(1-\epsilon)2t_k)}(y)\,|f(y)|\,dy
\leq \frac{C\,\log\frac{1}{\epsilon}}{\beta}\int_\Omega
|f(y)|\,dy.
\end{align*}
Thus we have proved the weak-type $1-1$ estimate in $(i)$. This together with the obvious inequality $\|\mathcal M(f)\|_{L^\infty(\Omega)}\leq \|f\|_{L^\infty(\Omega)}$ and the Marcinkiewicz interpolation 
lemma (see Theorem~5 in \cite[Page~21]{St})  yields the strong-type $p-p$ estimate in $(ii)$. Alternatively, $(ii)$ can be obtained by using the same arguments as in the proof of \cite[Theorem~2.8.2]{Z}. 
\end{proof}

\section{quasi-distance and  space of homogeneous type}\label{sec:quasi-distance}

In this section we will introduce a quasi-distance $d$ induced by sections of solutions $\phi$ to the Monge-Amp\`ere equation in $\Omega$. Moreover,  we show that $(\overline{\Omega}, d, 
\mu)$ is a space of homogeneous type, where $\mu := \det D^2\phi \,dx\,$ is the Monge-Amp\`ere measure. We begin with the following simple lemma.

\begin{lemma}\label{doubling-engulfing}
Assume that the convex domain $\Omega$ and the convex function $\phi$ satisfy \eqref{global-tang-int}--\eqref{global-sep}. For all  $x\in \overline{\Omega}$ and $t>0$, we have
\begin{myindentpar}{1cm}
(i) if $y\in S(x,t)$, then $S(y,t)\subset S(x, \theta_{\ast}^2 t)$;\\
(ii) $|S(x, 2t)| \leq C \, |S(x,t)|$.
\end{myindentpar}
Here $\theta_{\ast}$  is the engulfing constant and $C$ depends  only on  $\rho,\lambda, \Lambda$ and $n$. 
\end{lemma}
\begin{proof}
If $y\in S(x,t)$, then  $x\in S(x,t)\subset S(y, \theta_{\ast} t)$ by Theorem~\ref{engulfing2}. By applying again the engulfing property, we obtain $S(y, \theta_{\ast} t)\subset S(x, \theta_{\ast}^2 t)$ which gives $(i)$.

To prove the doubling property $(ii)$, let $c_0$ be the universal constant in  Corollary~\ref{volume_growth}. If $2 t\leq c_0$, then $(ii)$ follows from the volume growth given by  Corollary~\ref{volume_growth}. Now assume $2 t>c_0$. Then we have
\begin{align*}
|S(x, 2t)| \leq |\Omega|\leq |B_{\frac{1}{\rho}}|=  C \, C_1 \big(\frac{c_0}{2}\big)^{\frac{n}{2}}
\leq C \, |S(x,\frac{c_0}{2})|
\leq  C \, |S(x,t)|,
\end{align*}
where the third inequality is by Corollary~\ref{volume_growth}. The proof is thus complete.
\end{proof}

Let $\Omega\subset \R^n$ be a convex set and $\phi\in C(\overline{\Omega})$ be a convex function. We define a function $d: \overline{\Omega}\times \overline{\Omega}\longrightarrow [0,\infty)$ by
\begin{equation}\label{quasi-distance}
d(x,y) := \inf{\big\{r>0: x\in S_\phi(y,r)\mbox{ and } y\in S_\phi(x,r)\big\}} \quad \forall x,y\in\overline{\Omega}.
\end{equation}
Also the induced $d$-ball with center $x\in\overline{\Omega}$ and radius $r>0$ is given by
\[
B_d(x,r) :=\{y\in \overline{\Omega}: d(x,y)<r\}.
\]
The next result is the boundary version of that in  \cite[Section~3]{AFT} where interior sections are considered.

\begin{theorem}\label{homogeneous-space}
Assume that the convex domain $\Omega$ and the convex function $\phi$ satisfy \eqref{global-tang-int}--\eqref{global-sep}. Let  $d: \overline{\Omega}\times\overline{\Omega} \longrightarrow [0,\infty)$ be defined by \eqref{quasi-distance}. Then the function $d$  satisfies
\begin{myindentpar}{1cm}
(i)  $d(x,y)=d(y,x)\,$ for all $\, x,y\in\overline{\Omega}$;\\
(ii) $d(x,y)=0\,$ if and only if $\, x=y$;\\
(iii) $d(x,y) \leq \theta_{\ast}^2 \, \big[ d(x,z) + d(z, y)\big]\,$ for all $\, x,y,z\in\overline{\Omega}$.
\end{myindentpar}
In addition, we have
\begin{equation}\label{topo-equivalent}
S_\phi(x, \frac{r}{2 \theta_{\ast}^2}) \subset B_d(x,r) \subset S_\phi(x, r)
\end{equation}
for all $x\in\overline{\Omega}$ and $r>0$. Here $\theta_{\ast}>1$  is the engulfing constant given by Theorem~\ref{engulfing2}.
\end{theorem}
\begin{proof}
The theorem follows from Lemma~1 and Lemma~2 in \cite{AFT} with $K :=\theta_{\ast}^2$ provided that the following four conditions are satisfied:
\begin{myindentpar}{1cm}
(a)  $\bigcap_{r>0} S_\phi(x,r) = \{x\}\,$ for every $x\in\overline{\Omega}$;\\
(b)  $\bigcup_{r>0} S_\phi(x,r) = \overline{\Omega}\,$ for every $x\in\overline{\Omega}$;\\
(c) for each $x\in \overline{\Omega}$, the map $r\mapsto S_\phi(x,r)$ is nondecreasing in $r$;\\
(d) for any $y\in  S_\phi(x,r)$, we have $ S_\phi(x,r)\subset  S_\phi(y,\theta_{\ast}^2 r)$ and $ S_\phi(y,r)\subset  S_\phi(x,\theta_{\ast}^2 r)$.
\end{myindentpar}

Observe that $(b)$ holds  by Lemma~\ref{global-gradient} and $(c)$ is obvious. On the other hand,  property $(d)$ is a consequence of Theorem~\ref{engulfing2} and 
Lemma~\ref{doubling-engulfing} (i). 

To verify $(a)$, it suffices to show that $\bigcap_{r>0} S_\phi(x,r) \subset \{x\}$. First, we consider the case $x$ is a boundary
point of $\Omega$. Then, by (\ref{small-sec}), we have

$$\bigcap_{r>0}S_{\phi}(x, r) \subset \bigcap_{r>0} B(x, C r^{1/4}) =\{x\}.$$
Now, consider the case $x$ is an interior point of $\Omega$. 
Let $\tilde{x}\in  \bigcap_{r>0} S_\phi(x,r)$. Then $\phi(\tilde x)< \phi(x) + \nabla\phi(x)\cdot (\tilde x - x) +r$ for every $r>0$. It follows that 
\[
\phi(\tilde x)=\phi(x) + \nabla\phi(x)\cdot (\tilde x - x),
\]
that is, the supporting hyperplane $z= \phi(x) + \nabla\phi(x)\cdot (y - x)$ touches  the graph of $\phi$ at both 
$x$ and $\tilde x$. Since $\Omega$ and $\phi$ satisfy \eqref{global-tang-int}--\eqref{global-sep}, $\phi$ is $C^{1,\alpha}$ on the boundary
$\p\Omega$ for all $\alpha\in (0,1)$ as observed in \cite[Lemma 4.1]{LS}. In fact, we have for all $x_{0}\in \p\Omega$
and for all $x$ in $\overline{\Omega}$ close to $x_{0}$,
$$\abs{\phi(x)-\phi(x_{0})-\nabla\phi(x_{0})\cdot (x-x_{0})}\leq C\abs{x-x_{0}}^2 (\log \abs{x-x_{0}})^2.$$
Consequently, by Caffarelli's Localization Theorem \cite{C1}, we know that $\phi$ is strictly convex in $\Omega$. 
Therefore, we infer that $\tilde x= x$ and so property $(a)$ is proved.
\end{proof}

It follows from properties $(i)-(iii)$ in Theorem~\ref{homogeneous-space} that $d$ is
a quasi-distance on $\overline{\Omega}$ and $(\overline{\Omega}, d)$ is a quasi-metric space. Moreover, as a consequence of Lemma~\ref{doubling-engulfing} (ii) and \eqref{topo-equivalent} we obtain the following doubling property for $d$-balls:
\begin{equation*}
|B_d(x, 2r )| \leq |S(x, 2r)|\leq C \,|S(x,\frac{r}{2 \theta_{\ast}^2})|\leq C\, |B_d(x, r)|\quad \mbox{for all $x\in \overline{\Omega}$ and $r>0$},
\end{equation*}
where $C$ depends only on $\rho,\lambda,\Lambda$ and $n$. Thus, $(\overline{\Omega}, d, 
|\cdot|)$ is a doubling quasi-metric space and hence it is a space of homogeneous type; see \cite[Remark on p. 67]{CW}. We refer readers to  
\cite{CW,DGL} for some results and analysis on this type of spaces.

\bibliographystyle{plain}

\end{document}